\documentclass[reqno,draft]{amsart}

\usepackage{color}
\usepackage{amssymb}
\usepackage{amsmath}
\usepackage{amsthm}
\usepackage{amsmath}
\usepackage{xypic}

\renewcommand\emptyset{\varnothing}

\newtheorem{theorem}{Theorem}[section]
\newtheorem{lemma}[theorem]{Lemma}
\newtheorem{corollary}[theorem]{Corollary}
\newtheorem{cor}[theorem]{Corollary}
\newtheorem{proposition}[theorem]{Proposition}
\newtheorem{prop}[theorem]{Proposition}

\theoremstyle{definition}

\theoremstyle{remark}
\newtheorem{example}[theorem]{Example}
\newtheorem{examples}[theorem]{Examples}
\newtheorem{remarks}[theorem]{Remarks}
\newtheorem{remark}[theorem]{Remark}

\newtheorem*{example-unnumbered}{Example}
\newtheorem*{claim}{Claim}

  \def\Q{\mathbb{Q}} \def\F{\mathbb{F}} \def\Z{\mathbb{Z}}  
\def\N{\mathbb{N}}  \def\l{\lambda} \def\ll{\langle} \def\rr{\rangle}
 \def\a{\alpha}   \def\bp{\begin{pmatrix}}

\def\sm{\setminus} \def\ep{\end{pmatrix}} \def\bn{\begin{enumerate}} 
   \def\en{\end{enumerate}}
\def\ba{\begin{array}} \def\ea{\end{array}}  
   \def\a{\alpha} \def\b{\beta} \def\ti{\widetilde}
\def\id{\operatorname{id}}  \def\im{\operatorname{Im}} 
  
\def\ker{\operatorname{Ker}}\def\be{\begin{equation}} \def\ee{\end{equation}} 
   
 \def\hom{\operatorname{Hom}}  
\def\aut{\operatorname{Aut}}  
 \def\dim{\operatorname{dim}}

    \def\fr12{\frac{1}{2}} \def\z12{\Z[\fr12]}

\def\ol{\overline}

\def\tipi{\ti{\pi}}

\renewcommand\epsilon{\varepsilon}

\title{Residual properties of graph manifold groups}

\dedicatory{To the memory of Bernard Perron \textup{(}1944--2008\textup{)}.}

\author{Matthias Aschenbrenner}
\address{University of California, Los Angeles, California, USA}
\email{matthias@math.ucla.edu}
\thanks{The first author was partially supported by National Science Foundation grant DMS-0556197.}

\author{Stefan Friedl}
\address{University of Warwick, Coventry, UK}
\email{s.k.friedl@warwick.ac.uk}

\numberwithin{equation}{section}

\date{April 20, 2010}

\begin{document}

\begin{abstract}
Let $f\colon M\to N$ be a continuous map between closed irreducible graph manifolds with infinite fundamental group. Perron and Shalen \cite{PS99} showed that
if $f$ induces a homology equivalence on all finite covers, then $f$ is in fact homotopic to a homeomorphism. Their proof used the statement that
 every graph manifold is finitely covered by a $3$-manifold  whose fundamental group is residually $p$ for every prime $p$.
We will show that this statement regarding  graph manifold groups is not true in general, but we will show how to modify the argument
of Perron and Shalen to recover their main result.
As a by-product we will determine all semidirect products $\Z \ltimes \Z^d$ which are residually $p$ for every prime $p$. 
\end{abstract}

\maketitle

\section{Introduction}


\noindent
We say that a group is a $p$-group if it is finite of order a power of $p$. (Here and in the rest of the paper, $p$ will denote a prime number.)
We say that a group $G$ is \emph{residually $p$} if for any non-trivial $g\in G$ there exists
a morphism $\a\colon G\to P$ to a $p$-group $P$ such that $\a(g)$ is non-trivial.
We say that  $G$ is \emph{virtually residually $p$} if there exists a finite index subgroup of $G$ which is residually $p$.
It is not restrictive to demand that this finite index subgroup is normal in $G$. (See Section~\ref{section:resp} below.)

Given an automorphism $\varphi\colon\Z^d\to \Z^d$ consider
the semidirect product $\Z\ltimes_\varphi \Z^d$ where $1\in \Z$ acts on $\Z^d$ by $\varphi$. (This is sometimes called the mapping torus of $\varphi$.) Recall that the underlying set of  $\Z\ltimes_\varphi \Z^d$ is the cartesian product $\Z\times \Z^d$ and the group operation is given by $$(i,x)\cdot (j,y)=(i+j,x+\varphi^i(y)) \qquad (i,j\in\Z,\ x,y\in\Z^d).$$
We first give a criterion for  $\Z \ltimes_\varphi \Z^d$ to be residually $p$.
Recall that $\varphi$ is  said to be \emph{unipotent} if $(\varphi-\id)^d=0$.
By considering the Jordan normal form of $\varphi$ one can easily show that  $\varphi$ is unipotent if and only if $\varphi$ has the single eigenvalue $1$. Finally recall that $\varphi$ is \emph{quasi-unipotent}
if some power $\varphi^k$ ($k>0$) of $\varphi$ is unipotent; equivalently, if all eigenvalues of $\varphi$ are roots of unity.

\medskip

The following theorem (shown in Section~\ref{sec:residual properties} of this paper) gives the complete picture of residually $p$ properties of mapping tori of automorphisms of $\Z^d$:

\begin{theorem}\label{thm:residual properties of mapping tori}\label{mainthm}
Let $\varphi$ be an automorphism of $\Z^d$, and let $G=\Z\ltimes_\varphi \Z^d$. 
Denote by $P_\varphi(t)\in \Z[t]$ the characteristic polynomial of $\varphi$. 
Then the following hold:
\begin{enumerate}
\item $G$
is residually $p$ if and only if for every irreducible factor $P(t)\in \Z[t]$ of $P_\varphi(t)$ we have $P(1) \equiv 0 \bmod p$;
\item  $G$
is residually $p$ for all $p$ if and only if $G$ is nilpotent, if and only if
$\varphi$ is unipotent;
\item  $G$
has a finite index subgroup which is residually $p$ for all $p$ if and only if
$G$ is virtually nilpotent, if and only if
$\varphi$ is quasi-unipotent.
\end{enumerate}
\end{theorem}


Our main motivation for proving this theorem is to study fundamental groups of graph manifolds.
In \cite{AF10} we  show that if $N$ is a Seifert fibered $3$-manifold, then $\pi_1(N)$ has a finite index subgroup which is residually $p$ for every $p$.  In \cite{AF10} we furthermore prove the following weaker statement regarding the fundamental groups of graph manifolds.

\begin{theorem}\label{thm:gmresp}
Let $N$ be a graph manifold. Then for every $p$ the group $\pi_1(N)$ is virtually residually $p$,
i.e., for every $p$ there exists a finite index subgroup of $\pi_1(N)$ which is residually $p$.
\end{theorem}

It is now a natural question  whether the fundamental groups of graph manifolds always admit a finite index subgroup which is residually $p$ for every $p$. 
However, we will see that the `first non-trivial' example of a graph manifold already gives a counterexample.
More precisely, recall that a closed orientable \emph{Sol-manifold} is either the union of two twisted $I$-bundles over the Klein bottle, or it is a torus bundle over $S^1$ such that the eigenvalues $\l_1,\l_2$ of the monodromy $H_1(T;\Z)\to H_1(T;\Z)$
are of the form $\{e^{-t},e^t\}$ for some $t>0$ (see \cite[Section~9]{BRW05} and \cite[p.~470]{Sc83b} for details).
In the latter case we have an isomorphism   $\pi_1(N)\cong \Z\ltimes_\varphi \Z^2$
where $\varphi$ has eigenvalues $\{e^{-t},e^t\}$.
 Note that a Sol-manifold is not Seifert fibered, but it is a graph manifold.
The following proposition is now an immediate corollary to Theorem~\ref{mainthm},~(2).

\begin{proposition} \label{prop:sol}
Let $N$ be a  Sol-manifold which is a torus bundle over $S^1$. Then $\pi_1(N)$ does not have a finite index subgroup which is residually $p$ for all $p$.
\end{proposition}

In \cite{PS99}, Perron and Shalen introduce the notion of a `weakly residually $p$-nilpotent' group.
We show in Section~\ref{section:resp} that in fact, a group is weakly residually $p$-nilpotent if and only if it is residually $p$.
With this observation  \cite[Proposition~0.3]{PS99} states that every graph manifold group has a finite index subgroup which is
residually $p$ for every $p$. The proof for that claim, however, has a gap (cf. \cite[p.~36]{PS99}).
In fact, the combination of Propositions~\ref{prop:sol} and \ref{prop:weakstrong}   shows that Sol-manifolds  are counterexamples to 
\cite[Proposition~0.3]{PS99}.

\medskip

Let $f\colon M\to N$ be a continuous map of $3$-manifolds. (In this paper, all manifolds are assumed to be compact and connected.) We say that $f$ is a \emph{covering homology equivalence} if  for any finite covering $\ti{N}\to N$
(not necessarily regular) the induced map $\ti{f}\colon\ti{M}\to \ti{N}$ is a $\Z$-homology equivalence.
Perron and Shalen \cite{PS99} proved the following theorem under the (as we saw, erroneous) assumption that  all graph manifold groups have a finite index subgroup which is
residually $p$ for every $p$.

\begin{theorem}\label{thm:ps}
Let $M$ and $N$ be 
closed irreducible orientable graph manifolds with infinite fundamental group. Then every  covering homology equivalence $M\to N$ is homotopic to a homeomorphism.
\end{theorem}


In Section~\ref{section:proofps} we will show how to modify the original proof of Theorem~\ref{thm:ps} to accommodate for the weaker information
(coming from Theorem \ref{thm:gmresp}) on the virtual properties of the fundamental groups of a graph manifold.
We refer to \cite{PS99} for an application of Theorem~\ref{thm:ps} to singularity theory (the `$\mu$-constant problem' in complex dimension $3$).

\begin{remark}
The conditions in Theorem \ref{thm:ps} cannot be removed, as the following two examples show:
\bn
\item The lens spaces $L(5,1)$ and $L(5,2)$ are well-known to be homotopy equivalent but not homeomorphic: every homotopy equivalence $L(5,1) \to L(5,2)$ is a covering homology equivalence but not homotopic to a homeomorphism.
    This shows that it is necessary to demand that $M$ and $N$ have infinite fundamental groups.
\item Let $K\subseteq S^3$ be a non-trivial knot with Alexander polynomial $1$. Let $M$ be the $0$-framed surgery along $K$ and let $N=S^1\times S^2$.
Gabai showed that $M$ and $N$ are not homeomorphic \cite[Corollary~5]{Ga86}.
The abelianization $\pi_1(S^3\sm \nu K)\to \Z=\pi_1(S^1\times D^2)$ gives rise to a degree one map $S^3\sm \nu K\to S^1\times D^2$ which is a homeomorphism on the boundary. (Here $\nu K$ denotes a tubular neighborhood of $K$ in $S^3$.)
By capping off the manifolds we obtain a degree one map $f\colon M\to N$  which can be seen to be  a covering homology equivalence,
but $f$ is not homotopic to a homeomorphism. This shows that we cannot drop the condition that  $M$ is a graph manifold.
\en
\end{remark}

We mention that \cite{De03} contains a generalization of Theorem~\ref{thm:ps} to covering homology equivalences between closed Haken manifolds of the same Gromov simplicial volume, with the case of simplicial volume $0$ corresponding exactly to Theorem~\ref{thm:ps}.
The proof in \cite{De03} focusses on the case of non-zero simplicial volume and does not use residual properties of graph manifold groups.


\medskip

We conclude the paper with a proof of  the following  variation on Theorem \ref{thm:ps}.

\begin{theorem}\label{thm:ps2}
Let $M$ and $N$ be 
closed and  irreducible $3$-manifolds with infinite fundamental groups.
Assume that $\pi_1(M)$  is residually finite solvable. Then every covering homology equivalence $M\to N$  is homotopic to a homeomorphism.
\end{theorem}

Note that the fundamental groups of fibered 3-manifolds are residually finite solvable. We do not know whether  fundamental groups of graph manifolds with infinite fundamental group
are residually finite solvable.


\section{Residual Properties of Mapping Tori of Automorphisms of $\Z^d$}
\label{sec:residual properties}

\noindent
In this section we will give a detailed discussion of residual properties of mapping tori of automorphisms of $\Z^d$.
Along the way we will in particular prove Theorem \ref{thm:residual properties of mapping tori}.
%
We begin with some remarks on variants of residual torsion-freeness and on modules over group rings.

\subsection{Residual $\pi$-torsion-freeness}
Let $\pi$ be a set of prime numbers. A non-zero integer will be called a {\em $\pi$-number}\/ if its set of prime divisors is contained in $\pi$. We say that a group $G$ is {\em $\pi$-torsion-free}\/ if for each $\pi$-number $k$ the equation $x^k=1$ only has the trivial solution $x=1$ in $G$.

\begin{examples} \label{ex:pi-numbers} \mbox{}

\begin{enumerate}
\item If $\pi=\emptyset$, then the only $\pi$-numbers are $\pm 1$, and every group is $\pi$-torsion-free.
\item If $\pi=\text{the set of all primes}$, then every non-zero integer is a $\pi$-number, and $\pi$-torsion-free simply means torsion-free.
\item If $p$ is a prime and $\pi_p$ is the set of all primes different from $p$, then the set of $\pi_p$-numbers is $\Z\setminus p\Z$, and the $\pi_p$-torsion-free groups are the groups in which each element either has infinite order or $p$-power order. A finitely generated group is residually $p$ if and only if it is residually $\pi_p$-torsion-free nilpotent (cf.~\cite[Window~6, Proposition~1.2]{LS03}).
\end{enumerate}
\end{examples}

The elements of a nilpotent group $H$ whose order is a $\pi$-number form a normal subgroup $\operatorname{tor}_\pi(H)$ of $H$, the $\pi$-torsion subgroup of $H$. So given a group $G$ and $n\geq 1$, we may define
$\gamma^\pi_n(G)$ as the inverse image of $\operatorname{tor}_\pi(G/\gamma_n(G))$ under the natural morphism $G\to G/\gamma_n(G)$, that is,
$$\gamma^\pi_n(G) = \{g\in G: \text{$g^k\in\gamma_n(G)$ for some $\pi$-number $k$}\}.$$
One sees easily that $\gamma^\pi_n(G)/\gamma^\pi_{n+1}(G)$ is a $\pi$-torsion-free abelian group. The group $G$ is residually $\pi$-torsion-free nilpotent if and only if $\bigcap_{n\geq 1}\gamma^\pi_n(G)=1$.

\subsection{Modules over group rings}
Let $k$ be a commutative Noetherian ring. We denote the $k$-algebra of Laurent polynomials in the indeterminate $t$ over $k$ by $k[t^{\Z}]=k[t,t^{-1}]$. The ring $R=k[t^\Z]$ is also Noetherian.
The augmentation morphism $\epsilon\colon R\to k$ is the $k$-algebra morphism given by $t\mapsto 1$. The kernel of $\epsilon$ is the augmentation ideal
$\omega=\omega(R):= (1-t)R$.

Given a multiplicative subset $S$ of $R$, we say that an $R$-module $M$ is $S$-torsion-free if $sx\neq 0$ for all $s\in S$ and $x\in M$ with $x\neq 0$.
The next lemma is a simple case of the Krull Intersection Theorem:

\begin{lemma}\label{lem:KIT}
Let $M$ be a finitely generated $R$-module, and set $S=\{1-r:r\in \omega\}$ and $N=\bigcap_{n\geq 1} \omega^nM$. Then
$$N = \{x\in M: \text{$sx=0$ for some $s\in S$}\}.$$
In particular, $M$ is $S$-torsion-free if and only if
$N=0$.
\end{lemma}
\begin{proof}
The inclusion ``$\supseteq$'' is trivial: if $r\in \omega$, $x\in M$, with $(1-r)x=0$, then
$x=rx=r^2x=\cdots=r^nx$ for every $n$, hence $x\in \bigcap_n \omega^nM=N$.
For the reverse inclusion, suppose $M=Ry_1+\cdots+Ry_m$, and let $x\in N$. So for each $n\geq 1$ we can take $r_{in}\in R$ with $x=(t-1)^n (r_{1n}y_1+\cdots+r_{mn}y_m)$. The $R$-module $R^m$ being Noetherian, there is some $n\geq 1$ and $a_1,\dots,a_{n-1}\in R$ with $r_{in}=\sum_{j=1}^{n-1} a_j r_{ij}$ for $i=1,\dots,m$.
Hence
$$x=(t-1)^n \sum_{i=1}^m r_{in}y_i = \sum_{j=1}^{n-1} a_j(t-1)^{n-j}\cdot x.$$
So $(1-r)x=0$ where $r=\sum_{j=1}^{n-1} a_j(t-1)^{n-j}\in\omega$.
\end{proof}

In the rest of this subsection we assume $k=\Z$. Given a set $\pi$ of prime numbers, we let
$\Z[\pi^{-1}]$ be the localization of $\Z$ at the (multiplicative) subset of $\pi$-numbers subring, i.e., the subring of $\Q$ consisting of all rational numbers whose denominators are $\pi$-numbers. We put $R[\pi^{-1}]:=(\Z[\pi^{-1}])[t^\Z]$. (So $R_\emptyset=R$.) The ring $R[\pi^{-1}]$ is a Noetherian UFD.
We also let
$$S_\pi := \{r\in R: \text{$\epsilon(r)$ is a $\pi$-number}\}.$$
Let $M$ be a $\pi$-torsion-free abelian group.
Then the natural morphism $M\to M[\pi^{-1}]:=M\otimes_\Z \Z[\pi^{-1}]$ is injective, and we identify $M$ with a subgroup of $M[\pi^{-1}]$ in this way.
Clearly if $M$ is an $R$-module, then $M[\pi^{-1}]$ can naturally be given the structure of an $R[\pi^{-1}]$-module making $M$ an $R$-submodule of $M[\pi^{-1}]$. Moreover, if $N$ is an $R$-submodule of the $R$-module $M$, then $N[\pi^{-1}]$ is an $R[\pi^{-1}]$-submodule of $M[\pi^{-1}]$ in a natural way. Also,
$(\omega(R)^n M)[\pi^{-1}] = \omega(R[\pi^{-1}])^n M[\pi^{-1}]$ for every $n\geq 1$.
Thus from Lemma~\ref{lem:KIT} we obtain the following criterion for finitely generated $R$-modules to be $S_\pi$-torsion-free:

\begin{corollary}\label{cor:KIT}
Suppose $M$ is a finitely generated $\pi$-torsion-free $R$-module. Then the following are equivalent:
\begin{enumerate}
\item $M$ is $S_\pi$-torsion-free;
\item $M[\pi^{-1}]$ is $S_\pi^*$-torsion-free, where
$$S^*_\pi := \{r\in R[\pi^{-1}]: \epsilon(r)=1\}=\{1-r:r\in\omega(R[\pi^{-1}])\};$$
\item $\bigcap_{n\geq 1} (\omega(R)^n M)[\pi^{-1}]=0$.
\end{enumerate}
\end{corollary}

\subsection{Residual properties of mapping tori}
Let $A$ be a finitely generated $\pi$-torsion-free abelian group (written additively)  and let $\varphi$ be an automorphism of $A$; we construe $A$ as a (left) module over $R=\Z[t^{\Z}]$ in the natural way.
Form the semidirect product $G=\Z\ltimes_\varphi A$. Then
$G$ fits into a short exact sequence
$$0\to A\to G\to t^{\Z}\to 1,$$
and we construe $A$ as a normal subgroup of $G$ in this way. The normal subgroups of $G$ contained in $A$ are precisely the $R$-submodules of $A$.

\begin{lemma} \mbox{} \label{lem:powers of omega vs lower central series}
For all $n\geq 1$, we have
\begin{enumerate}
\item  $\omega^{n-1}A \geq \gamma_{n+1}(G) \geq \omega^n A$, and
\item  $\gamma^\pi_{n+1}(G)=(\gamma_{n+1}(G))[\pi^{-1}]\cap A$.
\end{enumerate}
\end{lemma}
\begin{proof}
Part (1) follows by an easy
induction on $n$, using that for $g\in G$, $a\in A$ we have
$[g,a]=g^{-1}a^{-1}ga=(1-\ol{g})a\in A$, where $\ol{g}$ is the image of $g$ under the natural projection $G\to t^\Z$. Part (2) is also easy to show, noting that by (1), for all $n\geq 1$, if $g\in G$ satisfies $g^k\in\gamma_{n+1}(G)$ then $\ol{g}^k=1$ and hence $\ol{g}=1$ (since $t^\Z$ is torsion-free), or equivalently, $g\in A$.
\end{proof}

In particular, by the previous lemma we have
$$\bigcap_{n\geq 1} (\gamma_{n+1}(G))[\pi^{-1}]=\bigcap_{n\geq 1} (\omega^n A)[\pi^{-1}],$$
and $G$ is residually $\pi$-torsion-free nilpotent if and only if $\bigcap_{n\geq 1} (\omega^n A)[\pi^{-1}]=0$. Therefore, by Corollary~\ref{cor:KIT}:

\begin{proposition}\label{prop:res pi-tf nilp}
The group $G$ is residually $\pi$-torsion-free nilpotent if and only if $A$ is $S_\pi$-torsion-free.
\end{proposition}

From now on suppose that $A$ is torsion-free, that is, $A\cong\Z^d$ for some $d$. Let $P_\varphi=\det(t\id-\varphi)\in\Z[t]$ be the characteristic polynomial of $\varphi$, construed as an automorphism of $A^*:=\Q^d$.
Next we show:

\begin{proposition}\label{prop:res pi-tf nilp, 2}
The group $G$ is residually $\pi$-torsion-free nilpotent if and only if no irreducible factor of $P_\varphi$ is in $S_\pi$.
\end{proposition}

\begin{proof}
This immediately follows from Proposition~\ref{prop:res pi-tf nilp} if the $R$-module $A$ has the form $A=R/P_\varphi R$.
In general,
we apply the structure theorem for finitely generated modules over the principal ideal domain $R^*:=\Q[t^\Z]$ to $A^*$ (considered as an $R^*$-module as usual) and obtain
$$A^* \cong R^*/P_1R^*\oplus \cdots \oplus R^*/P_lR^*$$
where the $P_i$ are the invariant factors of $\varphi$:
$P_i\neq 1$, $P_i\in \Z[t]$ monic, $P_i|P_{i+1}$, and $P_\varphi=P_1\cdots P_l$.
Then
$A$ has an $R$-submodule, of finite index in $A$ (as an abelian group), which is isomorphic to
$R/P_1R \oplus \cdots \oplus R/P_lR$.
Clearly, if $B$ is any finite-index $R$-submodule of $A$, then
$A$ is $S_\pi$-torsion-free if and only if $B$ is $S_\pi$-torsion-free. (We have an injective $R$-module morphism
$x\mapsto kx\colon A\cong \Z^d\to B$, where $k=[A:B]$.)
Thus $A$ is $S_\pi$-torsion-free if and only if each $R$-module $R/P_iR$ is $S_\pi$-torsion-free.
\end{proof}

The characterization given in Proposition~\ref{prop:res pi-tf nilp, 2} gives rise to a simple algorithm to decide, given a matrix representing $\varphi$ and a computable set $\pi$ of prime numbers, whether $G$ is residually $\pi$-torsion-free nilpotent. (An algorithm to decide whether a given finitely generated metabelian group is residually nilpotent is given in \cite{BCR94}.)

\medskip

By Examples~\ref{ex:pi-numbers} we also have to following corollary, which, in particular, proves  part  (1) of Theorem~\ref{thm:residual properties of mapping tori}.

\begin{corollary} \mbox{} \label{cor:res cor}

\begin{enumerate}
\item $G$ is residually nilpotent if and only if $\epsilon(P)\neq\pm 1$ for every irreducible factor $P$ of $P_\varphi$;
\item $G$ is residually torsion-free nilpotent if and only if $P_\varphi$ is a power of $t-1$;
\item $G$ is residually $p$ if and only if $\epsilon(P)\in p\Z$ for all irreducible $P$ dividing $P_\varphi$.
\end{enumerate}
\end{corollary}

This in turn now yields part (2) of Theorem~\ref{thm:residual properties of mapping tori}:

\begin{corollary}\label{cor:res cor, 2}
The following are equivalent:
\begin{enumerate}
\item $\varphi$ is unipotent;
\item $G$ is nilpotent;
\item $G$ is residually torsion-free nilpotent;
\item $G$ is residually $p$ for every $p$;
\item $G$ is residually $p$ for infinitely many $p$.
\end{enumerate}
In this case, $G$ is nilpotent of class at most $d+1$.
\end{corollary}
\begin{proof}
The equivalence of statements (1), (3), (4) and (5) is immediate from the previous corollary. The implication (2)~$\Rightarrow$~(3) is obvious ($G$ is torsion-free), and (1)~$\Rightarrow$~(2), as well as the addendum, are a consequence of Lemma~\ref{lem:powers of omega vs lower central series},~(1).
\end{proof}

Now fix a prime $p$. We let $\ol{A}:=A/pA\cong \F_p^d$, and we denote by $\ol{\varphi}$ the automorphism of $\ol{A}$ induced by $\varphi$.
We conclude this section with a short discussion on the relation between the residual properties of $G=\Z \ltimes_\varphi A$
and $\ol{G}:=\Z\ltimes_{\ol{\varphi}} \ol{A}$. This discussion is independent of Theorem~\ref{thm:residual properties of mapping tori}.
We recall a well-known fact:

\begin{lemma}\label{lem:res p mod p}
The following are equivalent:
\begin{enumerate}
\item $\ol{G}:=\Z\ltimes_{\ol{\varphi}} \ol{A}$ is residually $p$;
\item $\ol{\varphi}$ is unipotent;
\item $\ol{\varphi}^{p^k}=\id$ for some integer $k>0$.
\end{enumerate}
\end{lemma}
\begin{proof}
The equivalence of (1) and (2) may be seen, for example, as in the proof of Proposition~\ref{prop:res pi-tf nilp, 2}.
The equivalence of (2) and (3) is a familiar characterization of unipotent matrices over fields of characteristic $p$.
\end{proof}

By this lemma and Corollary~\ref{cor:res cor},~(3) we obtain:

\begin{corollary}\label{cor:res p mod p}
If $\ol{G}$ is residually $p$, then so is $G$.
\end{corollary}

\begin{example}
In general, the implication in the previous corollary cannot be reversed. For example, consider
$\varphi=\left(\begin{smallmatrix} 0 & 1 \\ 1 & a\end{smallmatrix}\right)$ where $a\in\Z$ is non-zero.
Then $P_\varphi=t^2-at-1$ is irreducible. Let $p\ne 2$ be a prime dividing $a$;
then $G$ is residually $p$, whereas $\ol{G}$ is not residually $p$. 
\end{example}

\subsection{Finite-index subgroups of mapping tori}
As before let $A=\Z^d$ and $\varphi\in\aut(A)$.
Given integers $k,n>0$, the subset $k\Z \times nA$ of $\Z\times A$ is the underlying set of a subgroup of $G=\Z\ltimes_\varphi A$, which we denote by $G_{k,n}$.
It is easy to see that  $$\Z\times A\to k\Z\times nA\colon (i,a)\mapsto (ki,na)$$ is a group isomorphism $\Z\ltimes_{\varphi^k} A\xrightarrow{\cong} G_{k,n}$.
If $n=1$, then $G_{k,n}\triangleleft G$; in fact, $G_{k,1}$ is the kernel of the natural morphism $G\to \Z\to\Z/k\Z$.
Clearly, if $H$ is a finite-index subgroup of $G$, then $G_{k,k}\leq H$ where $k=[G:H]$.
These remarks show the following well-known lemma:

\begin{lemma}
For every $p$, the group $G$ is residually $p$-by-finite cyclic.
\end{lemma}
\begin{proof}
Fix $p$; we employ the notation introduced at the end of the last subsection.
Let $k$ be the order of the automorphism $\ol{\varphi}$ of $\ol{A}$ induced by $\varphi$. Then $\ol{\varphi^k}=\id$, hence $G_{k,1}\cong \Z\ltimes_{\varphi^k} A$ is residually $p$ by Lemma~\ref{lem:res p mod p} and Corollary~\ref{cor:res p mod p}, and $G/G_{k,1}\cong\Z/k\Z$.
\end{proof}

In contrast to this, we have:

\begin{proposition}
The following are equivalent:
\begin{enumerate}
\item $G$ has a finite-index subgroup which is residually $p$ for every $p$;
\item $G$ is virtually nilpotent;
\item $\varphi$ is quasi-unipotent.
\end{enumerate}
\end{proposition}
\begin{proof}
The equivalence of (1) and (2) follows from Corollary~\ref{cor:res cor, 2}.
If $\varphi^k$ is unipotent ($k>0$), then the finite-index normal subgroup $G_{k,1}$ of $G$ is isomorphic to $\Z\ltimes_{\varphi^k} A$ and hence nilpotent, again by Corollary~\ref{cor:res cor, 2}. This shows (3)~$\Rightarrow$~(2). For the converse, let $H$ be a finite-index nilpotent subgroup of $G$. Then $G_{k,k}\leq H$, where $k=[G:H]$, hence $G_{k,k}$ is nilpotent; since $G_{k,k}\cong \Z\ltimes_{\varphi^k} A$, we see (Corollary~\ref{cor:res cor, 2} once again) that $\varphi^k$ is unipotent.
\end{proof}

This shows part (3) of Theorem~\ref{thm:residual properties of mapping tori}.

\section{Residually $p$ Equals Weakly Residually $p$-Nilpotent}\label{section:resp}

\noindent
Let $G$ be a group and let  $H$ be a subgroup of  $G$.
The \emph{core} of $H$ in $G$ is defined as
$$H_G = \bigcap_{g\in G} g^{-1}Hg,$$
i.e.,  $H_G$ is the largest normal subgroup of $G$ contained in $H$.
Note that if $H$ is of finite index in $G$, then so is $H_G$, by the following standard group theory fact (sometimes attributed to Poincar\'e):

\begin{lemma} \label{Finite index}
Let $H$ be a subgroup of finite index in $G$. Then $H$ has only finitely many conjugates in $G$.
\end{lemma}

\begin{proof}
The stabilizer of $H$ under the natural action of $G$ on the set of subgroups of $G$ via conjugation is the normalizer $N=N_G(H)$ of $H$ in $G$. Since $H$ has finite index in $G$, so does $N$, and the set of conjugates of $H$ in $G$ has cardinality $[G:N]$.
\end{proof}

A subgroup $H$ of $G$ is called {\it subnormal}\/  if there is a finite chain of subgroups of $G$, each one normal in the next, beginning at $H$ and ending at $G$:
\begin{equation}\label{Chain}
H = H_0 \triangleleft H_1 \triangleleft \cdots \triangleleft H_n=G.
\end{equation}
If $H$ is a subnormal subgroup of $G$ of finite index, then this finite chain can be chosen such that  each factor group $H_{i+1}/H_i$ is simple. In particular,
if $H$ is of $p$-power index, then the chain \eqref{Chain} can be chosen such that each factor $H_{i+1}/H_i$ is isomorphic to $\Z/p\Z$.

\medskip

Perron and Shalen \cite[p.~2]{PS99} say that $G$ is {\it weakly residually $p$-nilpotent} if for any $g\in G$ with $g\neq 1$ there exists a finite chain of subgroups of $G$
\[ H = H_0 \triangleleft H_1 \triangleleft \cdots \triangleleft H_n=G\]
with the following properties:
\bn
\item each group $H_i$ is normal in $H_{i+1}$,
\item for any $i$ the group $H_{i+1}/H_i$ is isomorphic to $\Z/p\Z$, and
\item $g$ is not contained in $H$.
\en
It follows from the above discussion that a group is
 weakly residually $p$-nilpotent exactly if for any $g\in G$ with $g\neq 1$ there exists a subnormal subgroup of $p$-power index in $G$ which does not contain $g$.
Clearly  residually $p$ implies weakly residually $p$-nilpotent. Somewhat less obviously, these two notions coincide:

\begin{prop}
\label{prop:weakstrong}
 If $G$ is weakly residually $p$-nilpotent, then $G$ is residually $p$.
\end{prop}

This proposition may be seen as a consequence of the well-known fact that being a $p$-group is a root property
(see \cite[p.~33]{Gru57}). For the reader's convenience we also give a quick proof of the proposition.
This follows immediately from the following lemma (which must surely be well-known):

\begin{lemma} \label{Key Lemma}
Suppose $H$ is a subnormal subgroup of $G$. Then
$[G:H_G]$ divides
$[G:H]^{[G:N_G(H)]}$. In particular, if $[G:H]$ is a power of $p$, then so is  $[G:H_G]$.
\end{lemma}

To prove Lemma~\ref{Key Lemma} we need an auxiliary observation.

\begin{lemma}
Let $H$ be a subnormal subgroup of finite index in $G$, and let $K$ be a subgroup of $G$. Then $[K : H \cap K]$ divides $[G : H]$, and hence $[G : H \cap K]$ divides $[G : H][G : K]$.
\end{lemma}

 Note that for arbitrary  subgroups $H$ and $K$ of $G$, in general one only can conclude $[G:H\cap K]\leq [G:H][G:K]$.

\begin{proof}
Consider a chain \eqref{Chain} as above. Intersecting with $K$ yields a chain
$$H \cap K = H_0 \cap K \leq H_1\cap K \leq \cdots \leq H_n\cap K = K.$$
For $i=1,\dots,n$, by the
Second Isomorphism Theorem applied to $H_i$ we have
$$[H_i \cap K : H_{i-1} \cap K] = [H_i \cap K : H_{i-1}\cap (H_i \cap K)] = [H_{i-1}(H_i \cap K):H_{i-1}],$$
and the order of its subgroup $H_{i-1}(H_i \cap K)/H_{i-1}$ divides the order of the group $H_i/H_{i-1}$,
hence $[H_i \cap K : H_{i-1} \cap K]$ divides $[H_i:H_{i-1}]$. Thus
$[K:H\cap K]=\prod_{i=1}^n [H_i\cap K:H_{i-1}\cap K]$
divides $[G:H]=\prod_{i=1}^n [H_i:H_{i-1}]$.
The second statement now follows from  $[G : H \cap K] = [G : K][K : H \cap K]$.
\end{proof}

By induction on $m$, the lemma yields:

\begin{cor}
If $G_1,\dots,G_m$ are subnormal subgroups of finite index in $G$, then
$[G:G_1\cap\cdots \cap G_m]$ divides $[G:G_1]\cdots [G:G_m]$.
\end{cor}

The corollary above and (the proof of) Lemma~\ref{Finite index} immediately imply Lemma~\ref{Key Lemma}.

\section{The Proof of Theorem~\ref{thm:ps}} \label{section:proofps}

\subsection{Discussion of the proof of Theorem~\ref{thm:ps}}

Let $M$ and $N$ be 
closed irreducible, orientable graph manifolds with infinite fundamental group
and let $f\colon M\to N$ be a covering homology equivalence.
If $M$ is a Seifert fibered space or a torus bundle, then  Theorem~\ref{thm:ps} is shown to hold in  \cite[Sections~5.1~and~5.2]{PS99}.
In these two cases no results on residual properties of fundamental groups are used.
Note that the case that $M$ is a torus bundle also follows immediately from Theorem~\ref{thm:ps2} since the fundamental groups of torus bundles are solvable.

The proof of Theorem \ref{thm:ps} for the remaining cases
builds on the following proposition:

\begin{proposition} \label{prop:tinj} \cite[Lemma~4.1.1]{PS99}
Let $f\colon M\to N$ be a covering homology equivalence between closed irreducible orientable graph manifolds with infinite fundamental group.
Then given any torus $T$ of the JSJ decomposition of $M$ the map $f_*\colon\pi_1(T)\to \pi_1(N)$ is injective.
\end{proposition}

\begin{remarks}
\mbox{}

\bn
\item
The proof provided in \cite{PS99} works under the  assumption that fundamental groups of graph manifolds are virtually residually $p$ for every $p$,
i.e., that there exists a finite index subgroup which is residually $p$ for every $p$.
In the following sections we provide a proof  of Proposition~\ref{prop:tinj} which is based on the  ideas of \cite[Section~4.1]{PS99}, but various changes to that treatment are in order to accommodate the
weaker information on virtual properties of graph manifold groups.
\item Note that in contrast to \cite{PS99} we do not exclude the case that $N$ is covered by a torus bundle.
\en
\end{remarks}

Now suppose that $M$ is neither a Seifert fibered space nor a torus bundle.
If $N$ is not finitely covered by a torus bundle, then the remainder of the proof of Theorem~\ref{thm:ps} provided in \cite[Sections~4.2~and~4.3]{PS99} carries over without any changes since it does not make use of any residual properties of graph manifold groups.
If $N$ is finitely covered by a torus bundle, then Theorem~\ref{thm:ps} is an immediate consequence of \cite[5.3.1~and~5.3.2]{PS99}.


\subsection{Proof of Proposition \ref{prop:tinj}, part I}

In this subsection,
we let $M$ and $N$ be 
closed irreducible orientable $3$-manifolds, and we let $f\colon M\to N$ be a covering homology equivalence. We also let $T$ be an incompressible torus in $M$ with the property that  $\im\{f_*\colon\pi_1(T)\to \pi_1(N)\}\cong \Z$.
The proof of the following lemma is partly based on ideas of  \cite[(4.1.3) and (4.1.4)]{PS99}.

\begin{lemma}\label{lem:choicea}
The following hold:
\bn
\item The torus $T$ is  separating in $M$.
\item There exists a component $A$ of $M$ cut along $T$ with the following property:
for any epimorphism $\a\colon\pi_1(N)\to G$ onto a finite group $G$ we have
\[ \im\left\{\pi_1(T) \to \pi_1(A)\xrightarrow{f_*}\pi_1(N)\xrightarrow{\a} G\right\}= \im\left\{\pi_1(A)\xrightarrow{f_*}\pi_1(N)\xrightarrow{\a} G\right\}.\]
\item For a component $A$ of $M$ as in \textup{(2)}  and every epimorphism $\a\colon\pi_1(N)\to G$ onto a finite group $G$
we  have  \[ H_1\left(\ker\left\{\pi_1(A)\xrightarrow{f_*}\pi_1(N)\xrightarrow{\a} G\right\};\Z\right)=\Z.\]
\en
\end{lemma}

The proof of this lemma will require the remainder of this section.

\medskip

We first prove (1). Since $ \im\{f_*\colon\pi_1(T)\to \pi_1(N)\}\cong \Z$ it follows that the map $T\to N$ factors up to homotopy through
a circle. In particular $f_*\colon H_2(T;\Z)\to H_2(N;\Z)$ is the trivial map.
On the other hand $f_*\colon H_2(T;\Z)\to H_2(N;\Z)$ factors as
$H_2(T;\Z)\to H_2(M;\Z)\xrightarrow{f_*} H_2(N;\Z)$ and the latter map is by assumption an isomorphism.
It follows that $T$ represents the trivial element in $H_2(M;\Z)$, i.e., $T$ is separating.

\medskip

We now turn to the proof of part (2).
We  denote the components of $M$ cut along $T$ by $A_1$ and $A_2$.

\begin{claim}
Let $\a\colon\pi_1(N)\to G$ be an epimorphism onto a finite group $G$. Then there exists an $i\in \{1,2\}$
such that
\[ \im\left\{\pi_1(T) \to \pi_1(A_i)\xrightarrow{f_*}\pi_1(N)\to G\right\}= \im\left\{\pi_1(A_i)\xrightarrow{f_*}\pi_1(N)\to G\right\}.\]
\end{claim}

This claim is apparently weaker than (2); however, given the claim, by taking product homomorphisms it is easy to verify that in fact there
exists an $i\in \{1,2\}$ such that for any epimorphism $\a\colon\pi_1(N)\to G$ onto a finite group $G$ we have
\[ \im\left\{\pi_1(T) \to \pi_1(A_i)\xrightarrow{f_*}\pi_1(N)\to G\right\}= \im\left\{\pi_1(A_i)\xrightarrow{f_*}\pi_1(N)\to G\right\}.\]
Thus, in order to show (2), it is enough to prove the claim.

\begin{proof}[Proof of the claim]
We  denote
by $q\colon N^\a\to N $ the cover corresponding to $\alpha$, and we denote by $p\colon M^\a\to M$ the induced cover.
Note that both $p$ and $q$ are regular covers.
Let $S$ be a  component of $p^{-1}(T)\subseteq M^\a$. Note that $S$ is an incompressible torus in $M^\a$.
Consider the following commutative diagram:
\[ \xymatrix{ \pi_1(S)\ar[d]^{p_*}\ar[r]^{f_*} & \pi_1(N^\a)\ar[d]^{q_*} \\ \pi_1(T)\ar[r]^{f_*}& \pi_1(N)}\]
Since the right vertical map is injective and since we assumed  that $\im\{f_*:\pi_1(T)\to \pi_1(N)\}\cong \Z$ it follows that $\im\{f_*:\pi_1(S)\to \pi_1(N^\a)\}\cong \Z$.
By assumption we also have that  $H_*(M^\a;\Z)\to H_*(N^\a;\Z)$ is an isomorphism. We can therefore  use  the argument of (1)
to conclude that any component of $p^{-1}(T)$  is separating in $M^\a$.

Given a component $S$ of $p^{-1}(T)$ we denote the two components of $M^\a$ cut along $S$ by
$C(S)$ and $D(S)$. We denote by $c(S)$ respectively $d(S)$ the number of components of $p^{-1}(T)$
contained in $c(S)$ respectively $d(S)$. Note that $c(S)$ and $d(S)$ are at least $1$ since $C(S)$ and $D(S)$
contain $S$. We will henceforth assume that we named $C(S)$ and $D(S)$ such
that $c(S)\leq d(S)$.

Among all components of  $p^{-1}(T)$ we take $S$ such that $c(S)$ is minimal. We claim that $c(S)=1$. For a contradiction, suppose that $c(S)>1$. Then let $S'\ne S$ be a component of $p^{-1}(T)$ contained in $C(S)$.
Recall that $S'$ is separating in $M^\a$, hence it in particular separates $C(S)$ into two components.
Note that one of the two components will have two boundary components, one of which is $S$, and the other component will have just one boundary component.
Denote by  $E(S')$ the component of $C(S)$ cut along $S'$ which does not contain $S=\partial C(S)$.
Note that $E(S')$ has fewer components of $p^{-1}(T)$  than $c(S)$ and that $E(S')$
is in fact a component of $M^\a$ split along $S'$. In particular we have $E(S')=C(S')$ or $E(S')=D(S')$ and in either case $c(S')<c(S)$, contradicting our choice of $S$. This contradiction shows that  $c(S)=1$.

Note that $c(S)=1$ implies that there exists an $i\in \{1,2\}$ such that $C(S)$ is a component of $p^{-1}(A_i)$,
and that component of $p^{-1}(A_i)$ has just one boundary component.  By regularity of the cover $p\colon M^{\a}\to M$ this means that all components
of $p^{-1}(A_i)$ contain exactly one component of $p^{-1}(T)$, which implies  that $b_0(p^{-1}(A_i))=b_0(p^{-1}(T))$.
On the other hand we have
\begin{align*}
b_0(p^{-1}(A_i))&= |G|\, / \, \left|\im\left\{\pi_1(T) \to \pi_1(A_i)\xrightarrow{}\pi_1(N)\to G\right\}\right|,\\
b_0(p^{-1}(T))&= |G|\, / \, \left|\im\left\{\pi_1(A_i)\xrightarrow{}\pi_1(N)\to G\right\}\right|.
\end{align*}
Note that
\[ \im\left\{\pi_1(T) \to \pi_1(A_i)\xrightarrow{f_*}\pi_1(N)\to G\right\}\subseteq \im\left\{\pi_1(A_i)\xrightarrow{f_*}\pi_1(N)\to G\right\}\]
and therefore
\[ \im\left\{\pi_1(T) \to \pi_1(A_i)\xrightarrow{f_*}\pi_1(N)\to G\right\}= \im\left\{\pi_1(A_i)\xrightarrow{f_*}\pi_1(N)\to G\right\}.\]
This concludes the proof of the claim.
\end{proof}

We finally turn to the proof of (3). Let $A$ be a component of $M$ cut along $T$ as in (2). For a contradiction, suppose $\a\colon\pi_1(N)\to G$ is a morphism onto a finite group $G$
such that  $H_1(\ker\{\pi_1(A)\xrightarrow{f_*}\pi_1(N)\to G\};\Z)\ne \Z$.
As before, we  denote
by $q\colon N^\a\to N$ the  cover corresponding to $\a$ and we denote by $p\colon M^\a\to M$ the induced cover. Since $p$ and $q$ are regular,  all components of
$p^{-1}(A)$ respectively $p^{-1}(T)$ are diffeomorphic.
We now pick   components $A^\a$ and $T^\a$ of $p^{-1}(A)$ and $p^{-1}(T)$
such  that $T^\a\subseteq \partial A^\a$.
It follows immediately from
\[ \im\left\{\pi_1(T) \to \pi_1(A)\xrightarrow{f_*}\pi_1(N)\to G\right\}= \im\left\{\pi_1(A)\xrightarrow{f_*}\pi_1(N)\to G\right\}\]
that $p^{-1}(A)$ and $p^{-1}(T)$ have the same number of components, in particular we have
that in fact $T^\a=\partial A^\a$.
Also note that
\[ H_1\left(\ker\left\{\pi_1(A)\xrightarrow{f_*}\pi_1(N)\to G\right\};\Z\right)=H_1(A^\a;\Z).\]
By our assumption we therefore have $H_1(A^\a;\Z)\ne \Z$.

Recall that $T^\a$ is the only boundary component of $A^\a$. We denote by $B^\a$ the other component of $M^\a$ split along $T^\a$.
It follows from well-known Poincar\'e duality arguments
(cf., e.g., \cite[Lemma~3.3]{PS99}) that for any prime $p$ we have
\[ \dim(\im\{H_1(T^\a;\F_p)\to H_1(A^\a;\F_p)\})=1.\]
In particular $b_1(A^a)\geq 1$.
Our assumption that $H_1(A^\a;\Z)\ne \Z$ now implies that there exists a prime $p$
such that $\dim(H_1(A^\a;\F_p))\geq 2$. Let $p$ be such a prime.
We write $I=\im\{H_1(T^\a;\F_p)\to H_1(A^\a;\F_p)\}$ and we let $C\leq H_1(A^\a;\F_p)$ such that $I\oplus C=H_1(A^\a;\F_p)$.
Now consider the Mayer-Vietoris sequence
\[ H_1(T^\a;\F_p)\to H_1(A^\a;\F_p)\oplus H_1(B^\a;\F_p)\to H_1(M^\a;\F_p).\]
It follows immediately that
$$C\to H_1(A^\a;\F_p)\to H_1(M^\a;\F_p)$$
is injective.
On the other hand the inclusion induced morphism $ H_1(T^\a;\F_p)\to  H_1(M^\a;\F_p)$  factors through  $H_1(T^\a;\F_p) \to  H_1(A^\a;\F_p)$,
in particular it follows that the image of $C$ in $H_1(M^\a;\F_p)$ is not contained in the image of $ H_1(T^\a;\F_p)$ in $H_1(M^\a;\F_p)$.
In particular we have
\be \label{eq:noteq} \im\{H_1(T^\a;\F_p)\to  H_1(M^\a;\F_p)\}\subsetneq \im\{H_1(A^\a;\F_p)\to  H_1(M^\a;\F_p)\}.\ee
We will now show that this leads to a contradiction to (2) for an appropriate choice of epimorphism from $\pi_1(N)$ to a finite group.

Note that $[\ker(\a),\ker(\a)]\ker(\a)^p$ is a normal subgroup of $\pi_1(N)$. We can therefore consider the epimorphism
\[\b\colon\pi_1(N)\to K:=\pi_1(N)/[\ker(\a),\ker(\a)]\ker(\a)^p.\]
We obtain  a short exact sequence
\[ 1\to \ker(\a)/[\ker(\a),\ker(\a)]\ker(\a)^p\to   K\to   G=\pi_1(N)/\ker(\a)  \to 1.\]
We have
$$\ker(\a)/[\ker(\a),\ker(\a)]\ker(\a)^p\cong H_1(\ker(\a);\F_p)=H_1(N^\a;\F_p).$$
In particular $K$ is a finite group which contains  $H_1(N^\a;\F_p)$.

Now note  the restrictions of the map $\b\circ f_*\colon\pi_1(M^\a)\to K$
to $\pi_1(T^\alpha)$ and $\pi_1(A^\a)$
factor as follows:
\[ \xymatrix{ \pi_1(T^\a)\ar[d]\ar[r] & \pi_1(A^\a)\ar[d]&\\
 H_1(T^\a;\F_p)\ar[r] & H_1(A^\a;\F_p)\ar[r]& H_1(M^\a;\F_p)\ar[r]^-{f_*} &H_1(N^\a;\F_p)\subseteq K.}\]
Our assumption that $f$ induces a homology equivalence $M^\a\to N^\a$ implies that $f_*\colon H_1(M^\a;\F_p)\to H_1(N^\a;\F_p)$ is an isomorphism.
It therefore follows from (\ref{eq:noteq}) that
 $\pi_1(T^\a)$ and $\pi_1(A^\a)$ have different images under the map $\b\circ f_*$.

Now consider the following commutative diagram where the horizontal sequences are exact:
\[ \xymatrix{ 1\ar[r]& \pi_1(T^\a)\ar[d]\ar[r] & \pi_1(T)\ar[d]\ar[r]&G \ar[d]^=& \\
1\ar[r]& \pi_1(A^\a)\ar[d]\ar[r] & \pi_1(A)\ar[d]\ar[r]&G \ar[d]^=& \\
0\ar[r]& H_1(N^\a;\F_p)\ar[r]&K\ar[r] & G\ar[r]&1.}\]
It now follows that  $\pi_1(T)$ and $\pi_1(A)$ have different images under the map $\b\circ f_*$.
But this contradicts our choice of $A$.

\medskip

This concludes the proof of (3) and hence the proof of Lemma~\ref{lem:choicea}. \qed

\subsection{Proof of Proposition~\ref{prop:tinj}, part II}

We first state a few lemmas before we turn to the proof of Proposition~\ref{prop:tinj}.

\begin{lemma}\label{lem:posb1}
Let  $A$ be a graph manifold with one boundary component such that $H_1(A;\Z)=\Z$.
Given an integer $m>0$ denote by $A_m$ the $m$-fold cyclic cover of $A$ corresponding
to $\pi_1(A)\to \Z \to \Z/m\Z$.
Then there exists an $m$ with $b_1(A_m)>1$.
\end{lemma}

\begin{proof} 
Let $\Delta\in \Z[t^\Z]$ be the Alexander polynomial corresponding to the Alexander module $H_1(A;\Z[t^\Z])$.
It follows from \cite[Proposition~2.9]{PS99} that $\Delta \ne 1$. By
\cite[Theorem~12.1]{EN85} all zeroes of $\Delta$ are roots of unity.
Let $m$ be such that $\Delta$ has a zero which is an $m$th root of unity. It now follows from \cite[p.~17]{Go77} or \cite[8.21]{BZ85} that $b_1(A_m)>1$.
\end{proof}

We recall the following well-known lemma.

\begin{lemma} \label{lem:biggerb1}
Let $\ti{X}\to X$ be a finite cover of manifolds. Then $H_1(\ti{X};\Q)\to H_1(X;\Q)$ is surjective, in particular $b_1(\ti{X})\geq b_1(X)$.
\end{lemma}

Finally we show a group-theoretic fact (through which Theorem~\ref{thm:gmresp} will enter the story):

\begin{lemma}\label{lem:order}
Let $\pi$ be a group with the property  that given any prime $p$ the group $\pi$ is virtually residually $p$.
Let $g\in \pi$ be an element of infinite order and let $m$ be a positive integer.
Then there exists an epimorphism $\a\colon\pi\to G$ onto a finite group $G$
such that $m$ divides the order of $\a(g)\in G$.
\end{lemma}

\begin{proof}
We start out with the following claim.

\begin{claim}
Given any prime $p$ and any integer $n>0$ there exists an epimorphism $\a:\pi\to G$ onto a finite group $G$
such that $p^{n}$ divides the order of $\a(g)\in G$.
\end{claim}

This claim easily yields the lemma: Write $m=p_1^{n_1}\cdot \dots \cdot p_k^{n_k}$ where $p_1,\dots,p_k$ are distinct primes and $n_1,\dots,n_k>0$.
By the claim, for each $i\in \{1,\dots,k\}$ we have an epimorphism $\a_i\colon\pi\to G_i$ onto a finite group $G_i$
such that $p_i^{n_i}$ divides the order of $\a_i(g)\in G_i$.
It is now clear that
\[ \a_1\times \dots \times \a_k\colon \pi \to G_1\times \dots \times G_k \]
has the desired properties.

\medskip

Thus it remains to prove the claim. Let $p$ and $n>0$ be given.
By assumption there exists a finite index normal subgroup $\ti{\pi}\subseteq \pi$ which is residually $p$.
Let $l\in\N$ be such that $\ll g\rr \cap \ti{\pi} = \ll g^l\rr$. We write $\ti{g}=g^l\in \ti{\pi}$. Since $\ti{g}$ has infinite order
and since $\ti{\pi}$ is residually $p$
there exists an epimorphism $\ti{\b}\colon\ti{\pi}\to P$ onto a $p$-group such that $\ti{\b}\left(\ti{g}^{p^{n-1}}\right)\neq 1$.
Now let $h_1=1,h_2,\dots,h_r\in \pi$ be such that $\pi=\bigcup_{i=1}^r \ti{\pi}h_i$. We consider
\[ \ba{rcl} \ti{\a}\colon\ti{\pi}&\to & P\times \dots \times P =P^r\\
x &\mapsto & \big(\ti{\b}(h_1xh_1^{-1}),\dots,\ti{\b}(h_rxh_r^{-1})\big).\ea \]
 Let $\ti{d}$ be the order of $\ti{\a}\left(\ti{g}\right)$. Evidently $\ti{d}$ is a power of $p$, i.e., $\ti{d}=p^m$ for some $m\in\N$.
 Note that $\ti{\b}\left(\ti{g}^{p^{n-1}}\right)$ is non-trivial, and therefore
  $\ti{\b}\left(\ti{g}^{p^{i}}\right)$ is non-trivial
for $i=0,\dots,n-1$. It follows that $\ti{d}\geq p^n$, hence $m\geq n$.

By construction, $\ker(\ti{\a})$ is normal in $\pi$.
We now identify $\tipi/\ker(\ti{\a})$ with a subgroup of $P^r$ and we denote the projection map $\pi\to \pi/\ker(\ti{\a})=:G$ by $\a$.  Note that $G$ is in general not a $p$-group.
We now get a commutative diagram
\[ \xymatrix{ \ti{\pi} \ar@{^`->}[d] \ar[r]^-{\ti{\a}} &\ti{\pi}/\ker(\ti{\a})\subseteq P^r \ar@{^`->}[d] \\
\pi\ar[r]^-\a & \pi/\ker(\ti{\a}) =G.}\]

Let $d$ be the order of ${\a}(g)$ in $G$.
Note that
$$\a(g)^{lp^{m}}=\a(g^l)^{p^m}=\a(\ti{g})^{p^m}=\ti{\a}(\ti{g})^{p^m}=1.$$
In particular $d|lp^m$.
We now write $d=kp^r$ where $k|l$ and $r\leq m$. We then have
\[ \ti{\a}(\ti{g})^{p^r}={\a}(\ti{g})^{p^r}=\a(g^l)^{p^r}=\a(g)^{lp^r}=\left(\a(g)^{kp^r}\right)^{\frac{l}{k}}=1.\]
In particular the order of $\ti{\a}(\ti{g})$, which equals $p^m$, divides $p^r$. It now follows that $p^n$ divides $d$.
This concludes the proof of the  claim.
\end{proof}

\subsection{Conclusion of the proof of Proposition~\ref{prop:tinj}}

We are now in a position to prove Proposition~\ref{prop:tinj}.
So assume we are given 
closed irreducible orientable graph manifolds $M$ and $N$ with infinite fundamental group, and a covering homology equivalence $f\colon M\to N$.
Let $T$ be a torus  of the JSJ decomposition of $M$.
It follows from \cite[(4.1.10)]{PS99}  that $f_*(\pi_1(T))$ cannot be trivial.
Since $N$ is irreducible with infinite fundamental group it follows that $\pi_1(N)$ is torsion-free,
 in particular we see that either $f_*\colon\pi_1(T)\to \pi_1(N)$ is injective or  $\im\{f_*:\pi_1(T)\to \pi_1(N)\}\cong \Z$.

For a contradiction, suppose  $\im\{f_*:\pi_1(T)\to \pi_1(N)\}\cong \Z$.
We denote by $A$ the component of $M$ cut along $T$ as in Lemma~\ref{lem:choicea},~(2).
Note that $A$ is a graph manifold with one boundary component and by Lemma~\ref{lem:choicea},~(3) we have $H_1(A;\Z)=\Z$.
Given an integer $m>0$, we denote by $A_m$ the $m$-fold cyclic cover corresponding
to $\pi_1(A)\to \Z \to \Z/m\Z$.
By Lemma~\ref{lem:posb1} we can find  $m$ such that $b_1(A_m)>1$.

Let $t\in \pi_1(T)$ be an element such that $f_*(t)$ generates  $\im\{f_*:\pi_1(T)\to \pi_1(N)\}\cong \Z$.
By Theorem \ref{thm:gmresp} and Lemma~\ref{lem:order} we can find  an epimorphism $\a\colon\pi_1(N)\to G$  onto a finite group $G$
such that $m$ divides the order of $\a(f_*(t))\in G$. We denote as above
by $q\colon N^\a\to N $ and  $p\colon M^\a\to M$ the corresponding induced covers and we let $A^\a$ be one of the components of $p^{-1}(A)$.

We denote by $d$ the order of $\a(f_*(t))\in G$ and we write
$$H=\im\left\{\pi_1(T)\to\pi_1(A)\xrightarrow{f_*} \pi_1(N) \xrightarrow{\alpha} G\right\}.$$
Note that $H\cong \Z/d\Z$. By our choice of $A$ we have
$$H=\im\left\{\pi_1(A)\xrightarrow{f_*} \pi_1(N)\xrightarrow{\alpha} G\right\}.$$
In particular the map $\pi_1(A)\to G$ factors through $\pi_1(A)\to \Z\to \Z/d\Z$.
It now follows that $A^\a=A_d$. Since $A_d$ is a finite cover of $A_m$,
Lemma~\ref{lem:biggerb1} yields $b_1(A^\a)\geq b_1(A_m)>1$;
this  contradicts Lemma~\ref{lem:choicea},~(3). \qed

\section{Proof of Theorem~\ref{thm:ps2}}\label{section:ps2}

\noindent
For the reader's convenience we recall the statement of Theorem~\ref{thm:ps2}.

\begin{theorem}
Let $M$ and $N$ be closed and  irreducible $3$-manifolds with infinite fundamental group.
Assume that $\pi_1(M)$  is residually finite solvable.
Let $f\colon M\to N$ be a covering homology equivalence.
Then $f$ is homotopic to a homeomorphism.
\end{theorem}

We will need the following purely group theoretic lemma.

\begin{lemma}\label{lem:completion}
Let $\phi\colon G\to H$ be a group morphism which, for every normal subgroup $\ti{H}$ of $H$ with finite solvable quotient $H/\ti{H}$,
induces an isomorphism $H_1(\ti{G};\Z)\to H_1(\ti{H};\Z)$, where $\ti{G}=\phi^{-1}(\ti{H})$.
Then  for every finite solvable group $S$, the map
\be \label{equ:prosolv} \hom(H,S) \xrightarrow{\phi^*} \hom(G,S)\ee
is a surjection.
Hence if in addition $\phi$ is surjective and $G$ is residually finite solvable, then $\phi$ is bijective.
\end{lemma}

\begin{proof}
We will show the first statement by induction on the derived length $\ell(S)$ of $S$.
If $\ell(S)=0$, then $S$ is the trivial group and there is nothing to prove.
Now suppose that the claim holds for any solvable group $S$ with $\ell(S)\leq n$.
Let $\a\colon G\to S$ be a morphism to a finite solvable group with $\ell(S)=n+1$; replacing $S$ with $\alpha(G)$ if necessary we may assume that $\alpha$ is onto.
We write $\ol{S}=S/S^{(n)}$ (where $S^{(n)}=\text{$n$th term of the derived series of $S$}$) and we denote the epimorphism $\alpha\colon G\to S\to S/S^{(n)}=\ol{S}$ by $\ol{\a}$.
By our induction assumption we know that $\ol{\a}=\ol{\b}\circ \phi$ for some epimorphism $\ol{\b}\colon H\to \ol{S}$.

We have the following commutative diagram:
\[  \xymatrix{ H_1(G;\Z[\ol{S}])\ar[d]^{\phi_*}\ar[r]^-\cong &\ker(\ol{\a})/[\ker(\ol{\a}),\ker(\ol{\a})]\ar[d]^{\phi}\\
H_1(H;\Z[\ol{S}])\ar[r]^-\cong &\ker(\ol{\b})/[\ker(\ol{\b}),\ker(\ol{\b})],}\]
where the horizontal maps are isomorphisms by Shapiro's Lemma (cf.~\cite[6.3.2, 6.3.4]{We94}), and where the right (and hence also the left) vertical map is an isomorphism by the hypothesis on $\phi$.
This diagram now gives rise to the following commutative diagram:
\[\xymatrix{
0\ar[r]& H_1(G;\Z[\ol{S}])\ar[d]^{\phi_*} \ar[r] & G/[\ker(\ol{\a}),\ker(\ol{\a})]\ar[d]^{\phi}\ar[r]^-{\ol{\a}} & \ol{S}\ar[d]^=\ar[r]& 1\\
0\ar[r]& H_1(H;\Z[\ol{S}]) \ar[r] & H/[\ker(\ol{\b}),\ker(\ol{\b})]\ar[r]^-{\ol{\b}} & \ol{S}\ar[r]& 1}\]
Note that the two horizontal sequences are exact and note that the left and right vertical map is an isomorphism.
We thus conclude that the middle map is also an isomorphism.
Now note that $\a$ factors as
$$ G\to G/[\ker(\ol{\a}),\ker(\ol{\a})]\to S;$$
we denote the second map by $\a$ as well.
We finally consider the following commutative diagram:
\[ \xymatrix{
G\ar[r]^{\phi} \ar[d] &H\ar[d] \\
G/[\ker(\ol{\a}),\ker(\ol{\a})]\ar[r]^{\phi}_\cong \ar[d]^{\a} & H/[\ker(\ol{\b}),\ker(\ol{\b})]\ar[dl] \\
S&}\]
It is now clear that  $\a=\b\circ \phi$ for some $\b\colon H\to S$.

Now assume $\phi$ is surjective and $G$ is residually finite solvable. For a contradiction suppose $g\in \ker\phi$ with $g\neq 1$.
Since $G$ is residually finite solvable we can find a morphism  $\a\colon G\to S$ to a finite solvable group
such that $\a(g)\neq 1$. By surjectivity of \eqref{equ:prosolv} we can find a homomorphism $\b\colon H\to S$ such that $\a=\b\circ \phi$.
Thus 
$(\b\circ \phi)(g)=\a(g)\neq 1$, hence $\phi(g)\neq 1$.
This contradiction shows  that $\phi$ is injective.
\end{proof}

Now assume we are given $f\colon M\to N$ as in the theorem.

\begin{claim}
The map $f_*\colon\pi_1(M)\to \pi_1(N)$ is an isomorphism.
\end{claim}
\begin{proof}
First assume that $N$ is orientable. Since $H_3(M;\Z)\to H_3(N;\Z)$ is an isomorphism
we see that $M$ is also orientable. It is well-known (cf.~\cite[Lemma~15.12]{He76}) that this implies that $f_*\colon\pi_1(M)\to \pi_1(N)$ is surjective. By the previous  lemma $f_*$ is in fact an isomorphism.

Now suppose that $N$ is not orientable.
Let
$\phi\colon\pi_1(N)\to \Z/2\Z$ be a morphism such that the corresponding cover  $\widehat{N}$ of $N$ is orientable, and let $\widehat{f}\colon\widehat{M}\to \widehat{N}$ be the induced map. Clearly $\widehat{f}$ is also a covering homology equivalence
and we deduce from the above that $\widehat{f}_*\colon\pi_1(\widehat{M})\to \pi_1(\widehat{N})$ is an isomorphism.
Consider the following commutative diagram:
\[ \xymatrix{ 1\ar[r] & \pi_1(\widehat{M})\ar[d]^{\widehat{f}_*}\ar[r] & \pi_1(M)\ar[r]\ar[d]^{f_*} & \Z/2\Z \ar[d]^=\ar[r] & 1\\
1\ar[r] & \pi_1(\widehat{N})\ar[r] & \pi_1(N)\ar[r]^\phi & \Z/2\Z \ar[r] & 1}\]
It now follows from the $5$-Lemma that $f_*\colon\pi_1(M)\to \pi_1(N)$ is an isomorphism.
\end{proof}


The following proposition now concludes the proof of the theorem.

\begin{proposition}
Let $M$ and $N$ be 
closed and  irreducible $3$-manifolds with infinite fundamental group.
Let $f\colon M\to N$ be a continuous map which induces an isomorphism on fundamental groups.
Then $f$ is homotopic to a homeomorphism.
\end{proposition}

\begin{proof}
Note that $f$ gives rise to a homotopy equivalence  $f\colon M\to N$. We denote the inverse homotopy equivalence by $g\colon N\to M$.
Scott
\cite{Sc83a} showed that any manifold which is homotopy equivalent to a Seifert fibered manifold is in fact Seifert fibered (cf.~also \cite{CJ94}).
By \cite[Section 5.3,~Theorem 6]{Or72} a homotopy equivalence of Seifert fibered spaces is in fact homotopic to a homeomorphism.
This shows that if one of $M$ or $N$ is a Seifert fibered space, then so is the other, and the homotopy equivalence is homotopic to a homeomorphism
(cf.~also \cite[p.~14]{PS99}).

Now suppose that neither $M$ nor $N$ is a Seifert fibered manifold. By the Torus Theorem (cf.~\cite{Sc80}) $M$ and $N$ contain an embedded torus
if and only if $\pi_1(M)=\pi_1(N)$ contains a subgroup isomorphic to $\Z^2$. If $M$ and $N$ contain an embedded torus, then they are Haken and the conclusion
of the proposition follows from \cite[Corollary~6.5]{Wa68}.

By geometrization the only remaining case is that $M$ and $N$ are hyperbolic.
The proposition now follows from Mostow Rigidity.
\end{proof}


\end{document}